\documentclass[a4paper,reqno,intlimits]{amsart}

\usepackage[T2A]{fontenc}
\usepackage[cp1251]{inputenc}
\usepackage[english]{babel}

\usepackage[noadjust]{cite}
\usepackage{upref}

\theoremstyle{plain}
\newtheorem{theorem}{Theorem}
\newtheorem{lemma}{Lemma}
\newtheorem{corollary}{Corollary}
\theoremstyle{remark}
\newtheorem{remark}{Remark}

\newcommand{\N}{\mathbb{N}}
\newcommand{\Z}{\mathbb{Z}}
\makeatletter
\newcommand{\Cset}[2][E]{C[#1,
  \if\relax\expandafter\@gobble#2\relax #2\else(#2)\fi]}
\makeatother
\renewcommand{\MR}[1]{\relax\ifhmode\unskip\fi}
\DeclareMathOperator{\Prob}{\mathsf{P}}

\begin{document}

\title[One class of continuous functions]%
{One class of continuous functions\\ related to Engel series\\
  and having complicated local properties}

\author[O. Baranovskyi]{Oleksandr Baranovskyi}
\address{Institute of Mathematics\\
  National Academy of Sciences of Ukraine\\
  3~Te\-re\-shchen\-kiv\-ska St.\\
  Kyiv\\
  01024\\
  Ukraine}
\email{baranovskyi@imath.kiev.ua}

\author[Yu. Kondratiev]{Yuri Kondratiev}
\address{Fakult\"at f\"ur Mathematik\\
  Universit\"at Bielefeld\\
  Postfach~100\,131\\
  Bielefeld\\
  33501\\
  Germany}
\email{kondrat@math.uni-bielefeld.de}

\author[M. Pratsiovytyi]{Mykola Pratsiovytyi}
\address{Faculty of Physics and Mathematics\\
  National Pedagogical Dragomanov University\\
  9~Pyrohova St.\\
  Kyiv\\
  01601\\
  Ukraine;
  Institute of Mathematics\\
  National Academy of Sciences of Ukraine\\
  3~Tereshchenkivska St.\\
  Kyiv\\
  01024\\
  Ukraine}
\email{pratsiovytyi@imath.kiev.ua}

\thanks{This research was partially supported by FP7-PEOPLE-IRSES program,
  grant no.~PIRSES-GA-2013-612669}



\begin{abstract}
  In the paper, we construct and study the class of continuous on $[0,
  1]$ functions with continuum set of peculiarities (singular, nowhere
  monotonic, and non-differentiable functions are among them).  The
  representative of this class is the function $y = f(x)$ defined by
  the Engel representation of argument:
  \begin{align*}
    x &= \sum_{n=1}^\infty
      \frac{1}{(2+g_1)(2+g_1+g_2)\ldots(2+g_1+g_2+\ldots+g_n)} = \\
    &= \Delta^E_{g_1g_2\ldots g_n\ldots},
    \quad \text{where} \quad g_n = g_n(x) \in \{ 0, 1, 2, \ldots \},
  \end{align*}
  and convergent real series
  \[
    \sum_{n=0}^\infty u_n = u_0 + u_1 + \ldots + u_n + r_n = 1, \quad
    \lvert u_n \rvert < 1, \quad
    0 < r_n < 1,
  \]
  by the following equality
  \[
    f(\Delta^E_{g_1(x)g_2(x)\ldots g_n(x)\ldots})
    = r_{g_1(x)} + \sum_{k=2}^\infty \biggl(
      r_{g_k(x)} \prod_{i=1}^{k-1} u_{g_i(x)} \biggr).
  \]
  We study local and global properties of function $f$: structural,
  extremal, differential, integral, and fractal properties.
\end{abstract}

\maketitle

\section{Introduction}

Most of continuous on the unit interval functions have complicated
local properties (infinite and even continuum set of
peculiarities)~\cite{Banach:1931:BKG,Mazurkiewicz:1931:FND,%
  Zamfirescu:1981:MMF,Kozyrev:1983:TTC:en}.  In particular, singular
functions (their derivative is equal to zero almost everywhere with
respect to Lebesgue measure), nowhere monotonic functions (they do not
have any arbitrary small monotonicity interval), and
non-differentiable functions (they do not have derivative in any
point) are among them.  This paper is devoted to such functions.  To
model and study them we need fine tools and methods.  Analytic
expressions for these functions contain infinite amount of operations
or limiting processes.  Series~\cite{Baranovskyi:2013:ROS:en},
infinite products, continued fractions and other are often used to
this end.  This is true at least for classic nowhere differentiable
Weierstra\ss~\cite{Turbin:1992:FMF:en}, Takagi~\cite{Takagi:1901:SEC},
Sierpi\'nski~\cite{Pratsiovytyi:2013:RIH:en},
Bush--Wunderlich~\cite{Bush:1952:CFD,Wunderlich:1952:USN} functions
etc., singular Cantor~\cite{Pratsiovytyi:1998:FPD:en},
Salem~\cite{Salem:1943:SSM,Chatterji:1963:CIM,Marsaglia:1971:RVI,%
  Takacs:1978:ICS}, Minkowski~\cite{Minkowski:1905:GZ},
Sierpi\'nski~\cite{Sierpinski:1916:EEF} functions etc.

Recently systems of functional
equations~\cite{Kalashnikov:2008:DFS:en,Pratsiovytyi:2010:OKS:en,%
  Pratsiovytyi:2011:OKN:en,Albeverio:2013:OCF:en,%
  Pratsiovytyi:2013:SSN:en}, iterated function systems, various system
of representation of numbers (systems of encoding of
numbers)~\cite{Pratsiovytyi:2002:FVO:en,Pratsiovytyi:2005:PDR:en,%
  Albeverio:2007:OSR,Baranovskyi:2007:TMP:en,Pratsiovytyi:2011:NMS:en,%
  Pratsiovytyi:2013:OSN:en,Pratsiovytyi:2013:FPF}, and automata with
finite memory (converters of digits from one representation to
another)~\cite{Pratsiovytyi:1989:NKP:en,Pratsiovytyi:2002:FVO:en,%
  Koval:2003:SHF:en,Panasenko:2008:FRH:en,Panasenko:2009:HBD:en} are
used to this end.

There exist some methodological problems in development of general as
well as individual theory of such functions.  First of all the reason
is an absence of effective means of their definition (description) and
tools for their study.

Ideas of theory of fractals (fractal geometry and fractal analysis),
i.e., self-similar, self-affine, scale-invariant properties can be
effectively used to this end.

Now several main directions in the study of local and global fractal
properties of functions are developed:
\begin{enumerate}
\item fractal characteristics of essential sets for function (for
  example, sets of peculiarities)~\cite{Turbin:1992:FMF:en,%
    Pratsiovytyi:1998:FPD:en};
\item properties of level sets of
  function~\cite{Pratsiovytyi:2013:RIH:en,Turbin:1992:FMF:en};
\item fractal properties of graphs~\cite{Panasenko:2006:FVO:en,%
    Panasenko:2008:FRH:en,Panasenko:2009:HBD:en,%
    Pratsiovytyi:2009:DFV:en};
\item preservation or transformation of fractal dimension by
  function~\cite{Albeverio:2004:FPD}.
\end{enumerate}

In this paper, we construct and study the infinite-parameter family of
continuous functions with complicated local properties: singular
(monotonic or non-monotonic), nowhere monotonic, non-differentiable or
almost everywhere non-dif\-fer\-en\-tiable functions are among them.  To
this end we use $E$-representation of real number, i.e., its encoding
by infinite alphabet in the form of Engel series (positive series such
that their terms are reciprocal to cumulative products of positive
integers)~\cite{Engel:1914:EZS,Erdos:1958:ESS,Renyi:1962:NAT,%
  Pratsiovytyi:2006:REZ:en,Hetman:2008:ZCS:en,Hetman:2010:VOS:en,%
  Baranovskyi:2011:PAM:en}.

Similar object related to representation of numbers with finite
alphabet and self-similar geometry ($s$-adic representation,
$Q$-representation) was studied in
papers~\cite{Pratsiovytyi:2011:OKN:en,Pratsiovytyi:2013:SSN:en}.  But
$E$-representation has an infinite alphabet and non-self-similar
geometry.  So some metric and probabilistic problems are essentially
complicated for $E$-representation than for self-similar and
N-self-similar representations.  In
papers~\cite{Pratsiovytyi:2011:OKN:en,Pratsiovytyi:2013:SSN:en}, they
found expression (not only estimation) for integral of function in
terms of parameters of initial system of functional equations and
described conditions for non-differentiability of function.

In our previous paper~\cite{Albeverio:2013:OCF:en} we studied
continuous functions such that they are solutions of infinite system
of functional equations with countable set of parameters related to
representation of real numbers by the first Ostrogradsky series.
Unlike~\cite{Albeverio:2013:OCF:en} in this paper we study properties
of level sets as well as scale-invariant and integral properties of
function.

\section{Object of study}

It follows~\cite{Pratsiovytyi:2006:REZ:en} from the known Engel
theorem~\cite{Engel:1914:EZS} that \emph{for any number $x \in (0, 1]$
  there exists a unique sequence $(g_n)$ of non-negative integers such
  that}
\begin{equation}
  \label{eq:Engel.series}
  x = \sum_{n=1}^\infty
    \frac{1}{(2+g_1)(2+g_1+g_2)\ldots(2+g_1+g_2+\ldots+g_n)}
  \equiv \Delta^E_{g_1g_2\ldots g_n\ldots}.
\end{equation}
The series~\eqref{eq:Engel.series} is called \emph{Engel series}, last
symbolic notation of number $x$ is called its
\emph{$E$-representation}, and $g_n = g_n(x)$ is $n$th symbol (digit)
of this representation.

Let us remark that $E$-representation of number is its encoding by
infinite alphabet $A \equiv \Z_0 = \{ 0, 1, 2, \ldots \}$.

If there exist $p \in \N$ such that $g_{m+np+j} = g_{m+j}$, $1 \leq j
\leq p$, for any $n \in \Z_0$, then they say that $E$-representation
has a period
\[
  g_{m+1}g_{m+2}\ldots g_{m+p}.
\]
It is written by $\Delta^E_{g_1g_2\ldots g_m(g_{m+1}g_{m+2}\ldots
  g_{m+p})}$.

The number is called \emph{$E$-rational} if its $E$-representation has
a period $(0)$.  The $E$-representation of such numbers has the
following form: $\Delta^E_{c_1c_2\ldots c_m(0)}$.

Let $(u_n)_{n=0}^\infty$ be an infinite sequence of real numbers
having the following properties (initial conditions):
\begin{align}
  \label{eq:un.init.conditions.sum1}
  &\sum_{n=0}^\infty u_n = u_0 + u_1 + \ldots + u_n + r_n
  = S_n + r_n = 1; \\
  \label{eq:un.init.conditions.abs1}
  &\lvert u_n \rvert < 1 \quad \text{for any $n \in \Z_0$}; \\
  \label{eq:un.init.conditions.rn.un}
  &0 < r_n \equiv \sum_{i=n+1}^\infty u_i = r_{n-1} - u_n
  = 1 - (u_0 + u_1 + \ldots + u_n) < 1, \quad n \in \Z_0.
\end{align}

It follows from property~\eqref{eq:un.init.conditions.sum1} that
sequence $(u_n)$ is infinitesimal, and it follows from
property~\eqref{eq:un.init.conditions.rn.un} that
\[
  0 < u_0 + u_1 + \ldots + u_n = S_n < 1
  \quad \text{for any} \quad n = 0, 1, 2, \ldots
\]
and series~\eqref{eq:Engel.series} has an infinite number of nonzero
terms.

Examples of various sequences $(u_n)$ are following:
\begin{enumerate}
\item $\frac{1}{u_0} = 2$, $\frac{1}{u_{n+1}} = \frac{1}{u_n} \left(
    \frac{1}{u_n} - 1 \right) + 1$, $n \in \Z_0$;
\item $u_n = \frac{1}{2^{n+1}}$, $n \in \Z_0$;
\item $u_n = \begin{cases}
    \frac{1}{2^{n+1}} & \text{if $n = 2k$}, \\
    0 & \text{if $n = 2k + 1$, $k \in \Z_0$};
  \end{cases}$
\item $u_0 = \frac{2}{3}$, $u_1 = - \frac{1}{6}$, $u_n =
  \frac{1}{2^n}$, $n = 2$, $3$, \ldots;
\item $u_{2(k-1)} = \frac{a}{2^k}$, $u_{2k-1} = \frac{1-a}{2^k}$, $k
  \in \N$, $a \in (1, 2)$.
\end{enumerate}

\begin{remark}
  \label{rem:un.insert.0}
  If sequence $(u_n)$ satisfies initial
  conditions~\eqref{eq:un.init.conditions.sum1}--\eqref{eq:un.init.conditions.rn.un},
  then putting finite or infinite number of zeroes between its terms
  we obtain new sequence satisfying initial
  conditions~\eqref{eq:un.init.conditions.sum1}--\eqref{eq:un.init.conditions.rn.un}
  too.
\end{remark}

The \emph{main object} of this study is the function
\begin{equation}
  \label{eq:f.series}
  y = f(x) = r_{g_1(x)} + \sum_{k=2}^\infty
    \biggl( r_{g_k(x)} \prod_{i=1}^{k-1} u_{g_i(x)} \biggr)
  \equiv \Delta_{g_1g_2\ldots g_n\ldots},
\end{equation}
where $g_n = g_n(x)$ is $n$th symbol of $E$-representation of the
number $x \in (0, 1]$.

It is evident that equality~\eqref{eq:f.series} does not define
function $f$ out of left-open interval $(0, 1]$.  Moreover,
\begin{gather*}
  f(1) = f(\Delta^E_{(0)}) = r_0 + r_0 u_0 + r_0 u_0^2 + \ldots
  = \frac{r_0}{1-u_0} = 1, \\
  f(\Delta^E_{c_1\ldots c_m(c)}) = r_{c_1} + \sum_{k=2}^m
    \biggl( r_{c_k} \prod_{i=1}^{k-1} u_{c_i} \biggr)
      + \frac{r_c}{1-u_c} \prod_{i=1}^m u_{c_i},
\end{gather*}
in particular,
\begin{equation}
  \label{eq:f.E.rational}
  f(\Delta^E_{c_1\ldots c_m(0)}) = r_{c_1} + \sum_{k=2}^m
    \biggl( r_{c_k} \prod_{i=1}^{k-1} u_{c_i} \biggr)
      + \prod_{i=1}^m u_{c_i}.
\end{equation}
Equality~\eqref{eq:f.E.rational} gives an expression for value of the
function in $E$-rational point.

Since any number $x \in (0, 1]$ has a unique $E$-representation, to
prove that function $f$ is well defined, it is enough to show that
series~\eqref{eq:f.series} is convergent for any sequence of
non-negative integers $(g_n)$.

There are negative numbers and zeroes among the terms $u_n$.  So, in
general, series~\eqref{eq:f.series} is not positive.  Thus consider
the series with general term
\[
  v_k = r_{g_k} \prod_{j=1}^{k-1} \lvert u_{g_j} \rvert.
\]

Since
\[
  r_{g_k} \leq r^*
  \equiv \max \{ r_0, r_1, \ldots, r_n, \ldots \}
\]
and
\[
  \lvert u_{g_k} \rvert \leq u^*
  \equiv \max \{ \lvert u_1 \rvert, \lvert u_2 \rvert, \ldots,
    \lvert u_n \rvert, \ldots \} < 1,
\]
we have $v_k \leq r^* (u^*)^{k-1} \equiv w_k$.

Using the direct comparison test we obtain that convergence of series
with general term $w_k$ implies convergence of series with general
term $v_k$.  Moreover,
\[
  \sum_{k=1}^\infty v_k \leq \sum_{k=1}^\infty w_k
  = \frac{r^*}{1-u^*}.
\]
Therefore, series~\eqref{eq:f.series} is absolutely convergent.  So
function $f$ is well defined.

We study local and global properties of function $f$: structural,
extremal, differential, integral and fractal properties, in
particular, ``symmetries'' of graph.

\section{Range of the function}

\begin{lemma}
  \label{lem:f.range}
  Value of the function $f$ belongs to closed interval $[0, 1]$.
\end{lemma}

\begin{proof}
  Let
  \[
    y_m = r_{g_1(x)} + \sum_{k=2}^m
    \biggl( r_{g_k(x)} \prod_{i=1}^{k-1} u_{g_i(x)} \biggr)
  \]
  be a partial sum of series~\eqref{eq:f.series}.

  We prove by induction that for any sequence $(g_n)$, $g_n \in \Z_0$,
  and any positive integer $m$ the following inequality holds:
  \begin{equation}
    \label{eq:ym.ineq}
    0 < y_m < 1.
  \end{equation}
  For $m = 1$, we have $y_1 = r_{g_1} \in (0, 1)$ by definition of
  $(u_n)$.  Consider
  \[
    y_2 = r_{g_1} + r_{g_2} u_{g_1}.
  \]
  If $g_1 = 0$, then
  \[
    0 < r_0 < y_2 = r_0 + r_{g_2} u_0 < r_0 + u_0 = 1,
  \]
  because $u_0 \in (0, 1)$.

  Let $g_1 > 0$.  If $u_{g_1} > 0$, then
  \[
    0 < r_{g_1} < y_2 < r_{g_1} + u_{g_1} = r_{g_1-1} < 1.
  \]
  If $u_{g_1} < 0$, then
  \[
    0 < r_{g_1-1} = r_{g_1} + u_{g_1} \leq y_2 \leq r_{g_1} < 1.
  \]
  Thus $0 < y_2 < 1$ for any sequence $(g_n)$.

  Suppose that $0< y_k < 1$ for any $(g_n)$ and consider $y_{k+1}$.

  Since
  \begin{equation*}
    y_{k+1} = r_{g_1} + u_{g_1} \biggl( r_{g_2} + \sum_{l=3}^{k+1}
      \biggl( r_{g_l} \prod_{i=1}^{l-1} u_{g_i} \biggr) \biggr)
    = r_{g_1} + u_{g_1} y_k^*,
  \end{equation*}
  where $y_k^*$ is a partial sum of series~\eqref{eq:f.series} for
  sequence $(g_2, g_3, \ldots, g_l, \ldots)$, by the inductive
  assumption, we have $0 < y_k^* < 1$.  Hence for $u_{g_1} > 0$ we
  obtain
  \[
    0 < r_{g_1} < y_{k+1} < r_{g_1} + u_{g_1} = r_{g_1-1} \leq 1,
  \]
  and for $u_{g_1} \leq 0$ we obtain
  \[
    0 < r_{g_1-1} < r_{g_1} + u_{g_1} \leq y_{k+1} \leq r_{g_1} < 1.
  \]
  Thus $0< y_{k+1} < 1$ for any sequence $(g_n)$.

  Then by principle of mathematical induction we have the double
  inequality~\eqref{eq:ym.ineq}. Passing to the limit in this
  inequality we obtain
  \[
    0 \leq y = \lim_{m\to\infty} y_m \leq 1.
    \qedhere
  \]
\end{proof}

\begin{lemma}
  The function $f$ satisfies the functional equation
  \begin{equation}
    \label{eq:f.func.eq}
    f(x) = r_{g_1(x)} + u_{g_1(x)} f(\omega(x)),
  \end{equation}
  where $\omega(x) = \omega(\Delta^E_{g_1(x)g_2(x)\ldots
    g_n(x)\ldots}) = \Delta^E_{g_2(x)g_3(x)\ldots g_n(x)\ldots} = x'$
  is a shift operator on symbols of $E$-representation of number.
\end{lemma}

\begin{proof}
  Proposition follows directly from fact that
  series~\eqref{eq:f.series} can be written in the form
  \[
    f(x) = r_{g_1(x)} + u_{g_1(x)} \biggl( r_{g_2(x)} + \sum_{k=3}^\infty
      r_{g_k(x)} \prod_{i=2}^{k-1} u_{g_i(x)} \biggr).
    \qedhere
  \]
\end{proof}

\begin{corollary}
  The function $f$ satisfies the infinite system of functional
  equations
  \[
    f(\delta_i(x)) = r_i + u_i f(x),
    \quad i = 0, 1, 2, \ldots,
  \]
  where $\delta_i(x) = \Delta^E_{ig_1(x)g_2(x)\ldots g_n(x)\ldots}$.
\end{corollary}

Let us remark that shift operator $\omega(x)$ is not a
piecewise-linear function of $x$ (for example, this is true for
expansion of numbers in the form of the L\"uroth series or the
$L$-representation~\cite{Zhykhareva:2008:ZCZ:en,%
  Zhykharyeva:2012:ENP,Pratsiovytyi:2013:TMP}), because
\[
  \omega(x) = x' = \Delta^E_{g_2(x)g_3(x)\ldots g_n(x)\ldots}
  = \frac{1}{2+g_2} + \frac{1}{(2+g_2)(2+g_2+g_3)} + \ldots.
\]
Thus relation between $x$ and $x'$ is more complicated.

Do not misunderstand shift operator $\omega(x)$ on symbols of
$E$-representation for operator (function)
\[
  \theta(x) = \Delta^E_{[g_1(x)+g_2(x)]g_3(x)g_4(x)\ldots}
  = [2+g_1(x)] x - 1.
\]
The latter is piecewise-linear.

\begin{theorem}
  \label{thm:f.range}
  The range of the function $f$ defined by
  equality~\eqref{eq:f.series} belongs to interval $(0, 1]$, and value
  $f(x)$ belongs to interval $(a(x), b(x)]$, where
  \begin{gather*}
    a(x) = \min \{ r_{g_1(x)}, r_{g_1(x)-1} \}, \quad
    b(x) = \max \{ r_{g_1(x)}, r_{g_1(x)-1} \}, \\
    r_{-1} \equiv u_0 + r_0 = 1.
  \end{gather*}
\end{theorem}

\begin{proof}
  Since equality~\eqref{eq:f.func.eq} holds and by
  Lemma~\ref{lem:f.range}
  \[
    0 \leq f(\omega(x)) \leq 1,
  \]
  for $u_{g_1(x)} > 0$ we have
  \[
    0 < a(x) = r_{g_1(x)} < f(x)
      \leq r_{g_1(x)} + u_{g_1(x)} = r_{g_1(x)-1} = b(x) \leq 1,
  \]
  and for $u_{g_1(x)} \leq 0$ we have $g_1(x) \geq 1$ and
  \[
    0 < a(x) = r_{g_1(x)-1} = r_{g_1(x)} + u_{g_1(x)} \leq f(x)
      \leq r_{g_1(x)} = b(x) < 1.
  \]
  Hence $a(x) < f(x) \leq b(x)$.  Since for $j \geq 0$
  \[
    0 < r_j < 1 \quad \text{and} \quad 0 < a(x) < f(x),
  \]
  we have $f(x) > 0$.
\end{proof}

\begin{corollary}
  \label{cor:f.range}
  If $u_{g_1(x)} > 0$, then $r_{g_1(x)} < f(x) \leq r_{g_1(x)-1}$; and
  if $u_{g_1(x)} \leq 0$, then $r_{g_1(x)-1} \leq f(x) \leq
  r_{g_1(x)}$.
\end{corollary}

\section{Continuity of the function}

To extend the definition of the function $f$ at the point $x = 0$, put
$f(0) = 0$.

\begin{theorem}
  \label{thm:f.continuous}
  The function $f$ is continuous at any point of interval $(0, 1)$,
  and it is right-continuous at the point $x = 0$, left-continuous at
  the point $x = 1$.
\end{theorem}

\begin{proof}
  Let $x_0$ be any point of $(0, 1)$.  To prove the continuity of the
  function at the point $x_0$ it is enough to show that
  \begin{equation}
    \label{eq:continuous.function}
    \lim_{x\to x_0} \lvert f(x) - f(x_0) \rvert = 0.
  \end{equation}

  If $x \neq x_0$, then there exists $m \in \N$ such that
  \[
    g_m(x) \neq g_m(x_0)
    \quad \text{but} \quad g_i(x) = g_i(x_0)
    \quad \text{for} \quad i < m.
  \]
  Then
  \begin{multline*}
    \lvert f(x) - f(x_0) \rvert
    = \biggl( \prod_{i=1}^{m-1} \lvert u_{g_i(x_0)} \rvert \biggr)
      \biggl\lvert r_{g_m(x)} + \sum_{k=m+1}^\infty
        \biggl( r_{g_k(x)} \prod_{j=m}^{k-1} u_{g_j(x)} \biggr) -{} \\
          - r_{g_m(x_0)} - \sum_{k=m+1}^\infty
            \biggl( r_{g_k(x_0)} \prod_{j=m}^{k-1} u_{g_j(x_0)} \biggr)
      \biggr\rvert.
  \end{multline*}
  Whence it follows that
  \[
    \lvert f(x) - f(x_0) \rvert
    < \prod_{i=1}^{m-1} \lvert u_{g_i(x_0)} \rvert
    \leq (u^*)^{m-1} \to 0 \quad (m \to \infty),
  \]
  where $u^* = \max \{ \lvert u_0 \rvert, \lvert u_1 \rvert, \ldots,
  \lvert u_n \rvert, \ldots \} < 1$.

  Since condition $x \to x_0$ is equivalent to condition $m \to
  \infty$, equality~\eqref{eq:continuous.function} holds.  Hence $f$
  is continuous at the point $x_0$.

  To prove that $f$ is left-continuous at the point $x = 1$ we use
  analogous arguments.

  Condition $x \neq 1$ means that there exists $m$ such that $g_m(x_0)
  \neq 0$.  At the same time $g_i(x) = 0$ for $i < m$.

  Condition $x \to 1-0$ is equivalent to condition $m \to \infty$.
  Since
  \[
    \lvert f(x) - f(1) \rvert
    < \prod_{i=1}^{m-1} \lvert u_{g_i(1)} \rvert
    \leq u_0^{m-1} \to 0 \quad (m \to \infty),
  \]
  we see that function $f$ is left-continuous at the point $x = 1$.

  Consider point $x = 0$.  Condition $x \to 0+0$ is equivalent to
  condition $g_1(x) \to \infty$.  Hence,
  \[
    \lvert f(x) - f(0) \rvert = f(x)
    \leq b(x) = \max \{ r_{g_1(x)}, r_{g_1(x)-1} \} \to 0,
  \]
  as $g_1(x) \to \infty$.  Thus $f$ is right-continuous at the point
  $x = 0$.
\end{proof}

\begin{corollary}
  The range of the function $f$ is a closed interval $[0, 1]$.
\end{corollary}

\section{Functional relations}

\begin{lemma}
  \label{lem:f.solution.func.eqs}
  The function $f$ defined by equality~\eqref{eq:f.series} is a unique
  solution of the system of functional equations
  \begin{equation}
    \label{eq:f.func.eqs}
    f(x) = r_i + u_i f(\omega(x)),
    \quad i = 0, 1, 2, \ldots,
  \end{equation}
  in the class of bounded functions defined at every point of $(0,
  1]$.
\end{lemma}

\begin{proof}
  From the fact that $f$ satisfies system~\eqref{eq:f.func.eqs} for
  any $x \in (0, 1]$ it follows that
  \begin{align*}
    f(x) &= f(\Delta^E_{g_1(x)g_2(x)\ldots g_n(x)\ldots})
    = r_{g_1} + u_{g_1} f(\Delta^E_{g_2(x)g_3(x)\ldots g_n(x)\ldots}) = \\
    &= r_{g_1} + r_{g_2} u_{g_1} + u_{g_1} u_{g_2}
      f(\Delta^E_{g_3(x)g_4(x)\ldots g_n(x)\ldots}) = \ldots = \\
    &= r_{g_1} + \sum_{k=2}^m r_{g_k} \prod_{i=1}^{k-1} u_{g_i}
      + \biggl( \prod_{i=1}^m u_{g_i} \biggr)
      f(\Delta^E_{g_{m+1}(x)g_{m+2}(x)\ldots g_{m+n}(x)\ldots}).
  \end{align*}

  Since $\prod\limits_{i=1}^m u_{g_i} \to 0$ for $m \to \infty$ and
  $f$ is bounded and defined at every point of $(0, 1]$ (in
  particular, expression $f(\Delta^E_{g_{m+1}(x)g_{m+2}(x)\ldots
    g_{m+n}(x)\ldots})$ is well defined), we see that remainder
  \[
    \biggl( \prod_{i=1}^m u_{g_i} \biggr)
      f(\Delta^E_{g_{m+1}(x)g_{m+2}(x)\ldots g_{m+n}(x)\ldots}) \to 0
  \]
  as $m \to 0$.  Thus solution of system~\eqref{eq:f.func.eqs} can be
  expressed by series~\eqref{eq:f.series} uniquely defining function
  $f$.
\end{proof}

\begin{corollary}
  The function $f$ is a unique continuous solution of system of
  functional equations~\eqref{eq:f.func.eqs}.
\end{corollary}

\begin{remark}
  Lemma~\ref{lem:f.solution.func.eqs} gives an equivalent definition
  of the function $f$ as a continuous solution of system of functional
  equations~\eqref{eq:f.func.eqs}.
\end{remark}

\begin{remark}
  Condition~\eqref{eq:un.init.conditions.sum1} provides that function
  $f$ is bounded and its range is an interval $[0, 1]$.  If
  condition~\eqref{eq:un.init.conditions.abs1} is not fulfilled, then
  system of functional equations~\eqref{eq:f.func.eqs} does not have
  solutions in the class of functions defined on $(0, 1]$, because
  series is divergent for some $x \in (0, 1]$.  If
  condition~\eqref{eq:un.init.conditions.rn.un} is not fulfilled, the
  function defined by~\eqref{eq:f.series} is not continuous.
  Accordance of sequences $(u_i)$ and $(r_i)$ provide continuity of
  solution, and if there is no accordance, then solution is an
  discontinuous function even if it exists.
\end{remark}

\section{Monotonicity intervals}

Let us recall the definition of useful notion of cylinder for
$E$-representation of number.

A \emph{cylinder of rank $m$ with base $c_1c_2\ldots c_m$} is the set
$\Delta^E_{c_1c_2\ldots c_m}$ of all numbers $x \in (0, 1]$ having
$E$-representation with first $m$ symbols $c_1$, $c_2$, \ldots, $c_m$
respectively, i.e.,
\[
  \Delta^E_{c_1c_2\ldots c_m} = \left\{ x \colon g_i(x) = c_i, \;
    i = \overline{1,m} \right\}.
\]

It is known~\cite{Pratsiovytyi:2006:REZ:en} that cylinder
$\Delta^E_{c_1c_2\ldots c_m}$ is a left-open interval with endpoints
\begin{align*}
  a_m &= \Delta^E_{c_1c_2\ldots c_{m-1}[c_m+1](0)}
  = \sum_{k=1}^m
    \frac{1}{(2+\sigma_1)(2+\sigma_2)\ldots(2+\sigma_k)}, \\
  b_m &= \Delta^E_{c_1c_2\ldots c_m(0)}
  = a_m + \frac{1}{(2+\sigma_1)(2+\sigma_2)\ldots
    (2+\sigma_m)(1+\sigma_m)},
\end{align*}
where $\sigma_k \equiv c_1 + c_2 + \ldots + c_k$.

As we can see the endpoints of cylinder $\Delta^E_{c_1c_2\ldots c_m}$
are $E$-rational points:
\[
  \Delta^E_{c_1c_2\ldots c_{m-1}[c_m+1](0)}
  \quad \text{and} \quad
  \Delta^E_{c_1c_2\ldots c_{m-1}c_m(0)}.
\]

\begin{remark}
  \label{rem:E.rational.cylinders}
  Any $E$-rational point is a common point of two cylinders of some
  rank belonging to the same cylinder of previous rank (we consider
  that $(0, 1]$ is a cylinder of zero rank).
\end{remark}

The cylinders have the following properties:
\begin{enumerate}
\item $\Delta^E_{c_1c_2\ldots c_m} = \bigcup\limits_{c=0}^\infty
  \Delta^E_{c_1c_2\ldots c_mc}$;
\item $\max \Delta^E_{c_1c_2\ldots c_m(c+1)} = \inf
  \Delta^E_{c_1c_2\ldots c_mc}$;
\item $\lvert \Delta^E_{c_1c_2\ldots c_m} \rvert =
  \dfrac{1}{(2+\sigma_1)(2+\sigma_2)\ldots(2+\sigma_m)(1+\sigma_m)}$;
\item $\dfrac{\lvert\Delta^E_{c_1c_2\ldots
      c_mc_{m+1}}\rvert}{\lvert\Delta^E_{c_1c_2\ldots c_m}\rvert} =
  \dfrac{1+\sigma_m}{(2+\sigma_{m+1})(1+\sigma_{m+1})}$;
\item For any sequence $(c_n)$, $c_n \in \Z_0$, the following equality
  holds:
  \[
    \bigcap_{m=1}^\infty \Delta^E_{c_1 c_2 \ldots c_m}
    \equiv \Delta^E_{c_1c_2\ldots c_m\ldots} = x \in (0, 1].
  \]
\end{enumerate}

\begin{lemma}
  \label{lem:f.const.cylinder}
  If $u_p = 0$, then function $f$ is constant on every cylinder
  $\Delta^E_{c_1c_2\ldots c_mp}$.
\end{lemma}

\begin{proof}
  Since $E$-representation of any number $x \in \Delta^E_{c_1c_2\ldots
    c_mp}$ has a form
  \[
    x = \Delta^E_{c_1c_2\ldots c_mpg_{m+2}g_{m+3}\ldots},
  \]
  we have
  \[
    f(x) = r_{c_1} + \sum_{k=2}^m
      \biggl( r_{c_k} \prod_{i=1}^{k-1} u_{c_i} \biggr)
      + r_p \prod_{i=1}^m u_{c_i} + 0,
  \]
  because $\prod\limits_{i=1}^{m+k} u_{g_i(x)} = 0$ for all $k \in
  \N$.
\end{proof}

\begin{corollary}
  If $(c_1, c_2, \ldots, c_m)$ is any tuple of non-negative integers
  such that
  \[
    u_{c_1} u_{c_2} \ldots u_{c_m} = 0,
  \]
  then function $f$ is constant on cylinder $\Delta^E_{c_1c_2\ldots
    c_m}$.
\end{corollary}

\begin{theorem}
  \label{thm:f.maxmin}
  The function $f$ takes the maximal and minimal values at endpoints
  of cylinder $\Delta^E_{c_1c_2\ldots c_m}$.  Moreover, if
  \[
    D_m \equiv \prod_{i=1}^m u_{c_i} \neq 0,
    \quad
    y_m = r_{c_1} + \sum_{k=2}^m \biggl( r_{c_k}
      \prod_{i=1}^{k-1} u_{c_i} \biggr),
  \]
  then for $D_m > 0$ we have
  \begin{align*}
    \max f(x) &= f(\Delta^E_{c_1c_2\ldots c_m(0)}) = y_m + D_m, \\
    \min f(x) &= f(\Delta^E_{c_1c_2\ldots c_{m-1}[c_m+1](0)}) = y_m;
  \end{align*}
  and for $D_m < 0$ we have
  \begin{align*}
    \max f(x) &= f(\Delta^E_{c_1c_2\ldots c_{m-1}[c_m+1](0)}) = y_m, \\
    \min f(x) &= f(\Delta^E_{c_1c_2\ldots c_m(0)}) = y_m + D_m;
  \end{align*}
\end{theorem}

\begin{proof}
  Since
  \[
    f(x) = y_m + D_m f(\omega^m(x)),
  \]
  where
  \[
    0 \leq f(\omega^m(x)) = r_{g_{m+1}(x)} + \sum_{k=m+2}^\infty
      \biggl( r_{g_k(x)} \prod_{i=m+1}^{k-1} u_{g_i(x)} \biggr) \leq 1,
  \]
  we see that for $D_m > 0$ function $f$ takes the maximal value if
  ${f(\omega^m(x)) = 1}$, i.e., if $x = \Delta^E_{c_1c_2\ldots c_m(0)}$,
  and minimal value if $f(\omega^m(x)) = 0$, i.e., if $x =
  \Delta^E_{c_1c_2\ldots c_{m-1}[c_m+1](0)}$.

  For $D_m < 0$ we have the opposite situation.
\end{proof}

\begin{corollary}
  \label{cor:f.change}
  Change in function $f$ on cylinder $\Delta^E_{c_1c_2\ldots c_m}$
  \[
    \mu_f(\Delta^E_{c_1c_2\ldots c_m})
    \equiv f(\Delta^E_{c_1c_2\ldots c_{m-1}c_m(0)})
      - f(\Delta^E_{c_1c_2\ldots c_{m-1}[c_m+1](0)})
  \]
  can be calculated by formula
  \begin{equation}
    \label{eq:f.change}
    \mu_f(\Delta^E_{c_1c_2\ldots c_m}) = \prod_{i=1}^m u_{c_i}.
  \end{equation}
\end{corollary}

\begin{corollary}
  \label{cor:f.const}
  If there are no zeroes among terms of sequence $(u_n)$, then
  function $f$ does not have constancy intervals.
\end{corollary}

\begin{proof}
  Indeed, suppose that under conditions of this proposition there
  exists interval $(a, b)$, where $f$ is constant. Then it is easy to
  find cylinder $\Delta^E_{c_1c_2\ldots c_m}$ completely belonging to
  $(a, b)$.  Hence, by Corollary~\ref{cor:f.change}, change in
  function $f$ on this cylinder as well as on interval $(a, b)$ is
  nonzero.  This contradiction completes the proof.
\end{proof}

\section{Lebesgue structure of the function}

The known Lebesgue theorem~\cite{Hennequin:1965:TPQ} states that any
function of bounded variation (in particular, probability distribution
function) can be represented in the form of linear combination
\begin{equation}
  \label{eq:Lebesgue.structure}
  F(x) = \alpha_1 F_d(x) + \alpha_2 F_{a.c.}(x) + \alpha_3 F_{s}(x),
\end{equation}
where $\alpha_i \geq 0$, $\alpha_1 + \alpha_2 + \alpha_3 = 1$, $F_d$,
$F_{a.c.}$, $F_{s}$ is discrete, absolutely continuous and singular
function respectively, i.e.,
\begin{enumerate}
\item $F_d$ is a function increasing only by jumps (jump function);
\item $F_{a.c.}(x) = \int\limits_{-\infty}^x F'(t) dt$;
\item $F_s$ is a continuous function but $F_s'(x) = 0$ Lebesgue-almost
  everywhere.
\end{enumerate}
Equality~\eqref{eq:Lebesgue.structure} is called Lebesgue structure of
function $F$.

Moreover, if one of the numbers $\alpha_i$ is equal to $1$, then
function is called pure.  If $\alpha_1 = 1$, then it is called pure
discrete, if $\alpha_2 = 1$, then it is called pure absolutely
continuous, if $\alpha_3 = 1$, then it is called pure singular (or
singularly continuous).

\begin{theorem}
  If all terms of sequence $(u_n)$ are non-negative, then $f$ is
  \begin{enumerate}
  \item\label{item:f.pdf.1} a probability distribution function on
    $[0, 1]$, moreover, this is a distribution function of random
    variable $\xi = \Delta^{E}_{\eta_1\eta_2\ldots\eta_k\ldots}$ such
    that its $E$-symbols $\eta_k$ are independent identically
    distributed random variables having the distribution $\Prob\{
    \eta_k = n \} = u_n$;
  \item\label{item:f.pdf.2} a strictly increasing function if $u_n >
    0$ for any $n \in \Z_0$;
  \item\label{item:f.pdf.3} a pure absolutely continuous or pure
    singular function.
  \end{enumerate}
\end{theorem}

\begin{proof}
  \ref{item:f.pdf.1}.~First of all let us prove that under conditions
  of the theorem function $f$ is non-decreasing.

  Let $x_1 < x_2$.  Then there exists $m \in \N$ such that
  \[
    g_m(x_1) > g_m(x_2)
    \quad \text{but} \quad
    g_i(x_1) = g_i(x_2) \quad \text{for} \quad i < m.
  \]
  Consider the difference
  \begin{multline*}
    f(x_2) - f(x_1)
    = \biggl( \prod_{i=1}^{m-1} u_{g_i(x_1)} \biggr)
      \biggl( r_{g_m(x_2)} - r_{g_m(x_1)} +{} \\
        {}+ \sum_{k=m+1}^\infty \biggl(
          r_{g_k(x_2)} \prod_{i=m}^{k-1} u_{g_i(x_2)} \biggr)
        - \sum_{k=m+1}^\infty \biggl(
          r_{g_k(x_1)} \prod_{i=m}^{k-1} u_{g_i(x_1)} \biggr)
      \biggr).
  \end{multline*}

  If $D = \prod\limits_{i=1}^{m-1} u_{g_i(x_1)} = 0$, then $f(x_2) -
  f(x_1) = 0$.  Let $D \neq 0$.  Then from condition $u_i \geq 0$ it
  follows that
  \[
    r \equiv r_{g_m(x_2)} - r_{g_m(x_1)}
    = u_{g_m(x_2)} + \ldots + u_{g_m(x_1)} \geq 0.
  \]

  If $r = 0$, i.e., $u_{g_m(x_2)+1} = \ldots = u_{g_m(x_1)} = 0$, then
  $\prod\limits_{i=m}^{m+j} u_{g_i(x_1)} = 0$ for all $j \in \N$. So,
  \[
    f(x_2) - f(x_1) = D \cdot \sum_{k=m+1}^\infty
      r_{g_k(x_2)} \prod_{i=m}^{k-1} u_{g_i(x_2)} \geq 0.
  \]

  If $r > 0$, then $r \geq u_{g_m(x_2)+1}$.  Hence,
  \begin{align*}
    f(x_2) - f(x_1) &\geq D \biggl( r_{g_m(x_2)} - r_{g_m(x_1)}
      - \sum_{k=m+1}^\infty r_{g_k(x_1)} \prod_{i=m}^{k-1} u_{g_i(x_1)}
      \biggr) \geq \\
    &\geq D (r_{g_m(x_2)} - r_{g_m(x_1)-1}).
  \end{align*}

  The last inequality follows from Corollary~\ref{cor:f.range} of
  Theorem~\ref{thm:f.range}.

  By $u_n \geq 0$ for all $n \in \N$ and $g_m(x_2) < g_m(x_1)$, we have
  \[
    r_{g_m(x_2)} - r_{g_m(x_1)-1} \geq 0.
  \]

  Therefore, $f(x_2) - f(x_1) \geq 0$ always.  Thus $f$ is
  non-decreasing continuous function taking value $0$ at $x = 0$ and
  value $1$ at $x = 1$, i.e., it is a probability distribution
  function on $[0, 1]$.

  Now let us show that expression for distribution function $F_\xi$ of
  random variable $\xi$ coincides with expression for function $f$.

  Find expression for distribution function $F_\xi = \Prob\{ \xi < x
  \}$.  An event $\{ \xi < x \}$ has the form
  \begin{multline*}
    \{ \xi < x \} = \{ \eta_1 > g_1(x) \}
      \cup \{ \eta_1 = g_1(x), \eta_2 > g_2(x) \} \cup \ldots \cup \\
      \cup \{ \eta_i = g_i(x), i = \overline{1,k-1}, \eta_k > g_k(x) \}
      \cup \ldots
  \end{multline*}
  From disjointness of events in the last union it follows that
  \[
    \Prob\{ \xi < x \} = \sum_{k=1}^\infty
      \Prob\{ \eta_i = g_i(x), i = \overline{1,k-1}, \eta_k > g_k(x) \}.
  \]
  Since symbols $\eta_k$ of $E$-representation of random variable
  $\xi$ are independent random variables, we have
  \begin{multline*}
    \Prob\{ \eta_i = g_i(x), i = \overline{1,k-1}, \eta_k > g_k(x) \} = \\
    = \sum_{j=g_k(x)+1}^\infty \Prob\{ \eta_k = j \}
      \prod_{i=1}^{k-1} \Prob\{ \eta_i = g_i(x) \}
    = r_{g_k(x)} \prod_{i=1}^{k-1} u_{g_i(x)}
  \end{multline*}
  and
  \[
    \Prob\{ \xi < x \} = r_{g_1(x)} + \sum_{k=2}^\infty
      \biggl( r_{g_k(x)} \prod_{i=1}^{k-1} u_{g_i(x)} \biggr).
  \]

  \ref{item:f.pdf.2}.~If $u_n > 0$ for all $n \in \N$, then $D > 0$
  and
  \[
    \sum_{k=m+1}^\infty r_{g_k(x_2)} \prod_{i=m}^{k-1} u_{g_i(x_2)} > 0.
  \]
  Hence,
  \[
    f(x_2) - f(x_1) > D (r_{g_m(x_2)} - r_{g_m(x_1)-1}) \geq 0,
  \]
  and thus $f$ is strictly increasing.

  \ref{item:f.pdf.3}.~Thus $f$ is a distribution function of random
  variable $\xi = \Delta^{E}_{\eta_1\eta_2\ldots\eta_k\ldots}$ with
  independent identically distributed $E$-symbols $\eta_k$.  By
  Theorem~\ref{thm:f.continuous}, $f$ is a continuous function.  So it
  is enough to prove that it cannot be a mixture of singular and
  absolutely continuous distributions.

  Let $x = \Delta^E_{g_1(x)g_2(x)\dots g_n(x)\dots}$ and let $t_1$,
  $t_2$, \ldots, $t_n$ be a fixed tuple of non-negative integers.
  Denote
  \[
    \bar\Delta_{t_1 t_2 \dots t_n}(x)
    = \Delta^E_{t_1t_2\dots t_ng_{n+1}(x)g_{n+2}(x)\dots}
  \]
  and for any subset $E$ of closed interval $[0, 1]$
  \begin{gather*}
    \bar\Delta_{t_1 t_2 \dots t_n}(E)
    = \{ u \colon u = \bar\Delta_{t_1 t_2 \dots t_n}(x), x \in E \}, \\
    T_n(E) = \bigcup_{t_1,t_2\dots,t_n}
      \bar\Delta_{t_1 t_2 \dots t_n}(E), \quad
    T(E) = \bigcup_n T_n(E).
  \end{gather*}

  Consider event $A = \{ \xi \in T(E) \}$.  Since $\eta_k$ are
  independent, we see that event $A$ generated by the sequence of
  random variables $\eta_k$ does not depend on all $\sigma$-algebras
  $\mathfrak{B}_m$ generated by $\eta_1$, \dots, $\eta_m$.  So $A$ is
  a residual event.  Thus, by Kolmogorov's $0$ and $1$ law, we have
  $\Prob(A) = 0$ or $\Prob(A) = 1$.

  Since $T(E) \supset E$, we see that from inequality $\Prob \{ \xi
  \in E \} > 0$ it follows that
  \[
    \Prob \{ \xi \in T(E) \} \geq \Prob \{ \xi \in E \} > 0.
  \]
  Thus $\Prob \{ \xi \in T(E) \} = 1$.

  Let us consider two cases:
  \begin{enumerate}
  \item There exists $E$ such that $\lambda(E) = 0$ and $\Prob \{ \xi
    \in E \} > 0$.

  \item For any set $E$ such that $\lambda(E) = 0$, we have $\Prob \{
    \xi \in E \} = 0$.
  \end{enumerate}

  In the first case from equality $\lambda(E) = 0$ it follows that
  $\lambda(T(E)) = 0$.  This means that there exists set $T(E)$ such
  that $\lambda(T(E)) = 0$ and $\Prob \{ \xi \in T(E) \} = 1$, i.e.,
  distribution of $\xi$ is pure singular by definition.

  In the second case distribution function of random variable $\xi$
  has an $N$-property.  This is equivalent to its absolute
  continuity~\cite{Natanson:1974:TFV:en}.
\end{proof}

\section{Conditions of nowhere monotonicity}

\begin{theorem}
  If sequence $(u_n)$ does not contain zeroes but contains negative
  terms, then function $f$ is nowhere monotonic on closed interval
  $[0, 1]$, i.e., it does not have any arbitrary small monotonicity
  interval.
\end{theorem}

\begin{proof}
  By Corollary~\ref{cor:f.const} from Theorem~\ref{thm:f.maxmin},
  function $f$ does not have constancy intervals.

  To prove that it does not have any monotonicity interval, it is
  enough to show that $f$ is not monotonic on any cylinder.  To this
  end, for any cylinder
  \[
    \Delta^E_{c_1c_2\ldots c_m} = (x_0, x_3],
  \]
  it is enough to give two points $x_1$ and $x_2$, where $x_1 < x_2$,
  such that values
  \[
    f(x_0), \; f(x_1), \; f(x_2), \; f(x_3)
  \]
  does not form a monotonic tuple of numbers.  Remark that it is
  enough even three points
  \[
    x_0, \; x_1, \; x_2 \quad \text{or} \quad x_1, \; x_2, \; x_3.
  \]
  Moreover, taking into account Theorem~\ref{thm:f.maxmin}, to prove
  monotonicity of $f$ on $\Delta^E_{c_1c_2\ldots c_m}$, it is enough
  to give cylinder of rank $m+1$ belonging to it such that change in
  function $f$ on this cylinder and change in function $f$ on
  $\Delta^E_{c_1c_2\ldots c_m}$ have different signs.

  It is evident that $\Delta^E_{c_1c_2\ldots c_mc}$, where $u_c < 0$,
  is such cylinder because change in function $f$ on
  $\Delta^E_{c_1c_2\ldots c_m}$ is equal to
  \[
    D_m = \prod_{i=1}^m u_{c_i},
  \]
  and on $\Delta^E_{c_1c_2\ldots c_mc}$ is equal to $D_m \cdot u_c$.
  They have different signs.

  If $D_m > 0$, then change in function $f$ is positive on cylinder
  $\Delta^E_{c_1c_2\ldots c_m}$ and negative on cylinder
  $\Delta^E_{c_1c_2\ldots c_mc}$.  If $D_m < 0$, then we have the
  opposite situation.
\end{proof}

\section{Extrema of the function}

Lemma~\ref{lem:f.const.cylinder} and Theorem~\ref{thm:f.maxmin} give
an exhaustive answer on the question on maximal and minimal value of
the function on cylinder.  It is enough to study function in the
endpoints of cylinders, that is in $E$-rational points (points having
the representation $\Delta^E_{c_1c_2\ldots c_mi(0)}$).

\begin{theorem}
  \label{thm:f.extremum}
  \textup1.~If $u_{i-1} u_i < 0$ for some $i$, then any point
  \[
    \Delta^E_{c_1c_2\ldots c_mi(0)},
    \quad \text{where} \quad
    D_m = \prod_{i=1}^m u_{c_1} \neq 0,
  \]
  is an extreme point of the function $f$.  Moreover, it is a maximum
  point if $D_m u_i > 0$, and minimum point if $D_m u_i < 0$.

  \textup2.~If $u_{i-1} u_i \geq 0$, any point
  \[
    \Delta^E_{c_1c_2\ldots c_mi(0)}
  \]
  is not an extreme point of the function $f$.
\end{theorem}

\begin{proof}
  1.~Let $D_m > 0$.  Then from Corollary of Theorem~\ref{thm:f.maxmin}
  follows that change in function $f$ on cylinder
  $\Delta^E_{c_1c_2\ldots c_m}$ is positive.

  If $u_i > 0$, then from the same corollary change in function is
  positive on cylinder $\Delta^E_{c_1c_2\ldots c_mi}$ and negative on
  cylinder $\Delta^E_{c_1c_2\ldots c_m[i-1]}$ lying to the right.  So
  common endpoint of these cylinders, that is point $x_i \equiv
  \Delta^E_{c_1c_2\ldots c_mi(0)}$, is a maximum point.

  If $u_i < 0$, then change in function $f$ is negative on cylinder
  $\Delta^E_{c_1c_2\ldots c_mi}$ and positive on cylinder
  $\Delta^E_{c_1c_2\ldots c_m[i-1]}$.  Thus point $x_i$ is a minimum
  point.

  If $D_m < 0$, then we obtain the result analogously.

  2.~If $u_{i-1} u_i = 0$, then using Lemma~\ref{lem:f.const.cylinder}
  we have that function is constant at least on one of cylinders
  \begin{equation}
    \label{eq:two.cylinders}
    \Delta^E_{c_1c_2\ldots c_mi}, \quad \Delta^E_{c_1c_2\ldots c_m[i-1]}.
  \end{equation}
  Thus, their common endpoint $\Delta^E_{c_1c_2\ldots c_mi(0)}$ is not
  extreme point.

  If $u_{i-1} u_i > 0$, then change in function $f$ has the same sign
  on both cylinders~\eqref{eq:two.cylinders}.  Thus, their common
  endpoint is not extreme point.
\end{proof}

\begin{corollary}
  If sequence $(u_n)$ has negative terms, then function $f$ has a
  countable set of extreme points, and they are $E$-rational numbers.
\end{corollary}

\begin{corollary}
  If sequence $(u_n)$ does not contain zeroes, but contains negative
  terms, then the set of extreme points of function $f$ is everywhere
  dense.
\end{corollary}

In fact, in this case extreme points exist in every cylinder.

\begin{remark}
  From Remark~\ref{rem:E.rational.cylinders} and
  Theorem~\ref{thm:f.extremum} follows that extreme points form empty
  set if $u_n \geq 0$ for all $n \in \N$ or countable subset of the
  set of $E$-rational points if at least one $u_i < 0$ exists and
  coincides with it if sequence $(u_n)$ is alternating.  Moreover,
  Theorem~\ref{thm:f.maxmin} describes the values of extrema.
\end{remark}

\section{Level sets of the function}

Let us recall that level set $y_0$ of function $f$ is a set
\[
  f^{-1}(y_0) = \{ x \colon f(x) = y_0 \}.
\]

If $u_n > 0$ for all $n \in \N$, then $f$ is a continuous strictly
increasing function as proved above.  Thus, any its level consists of
one point.

If $u_n \geq 0$ for all $n \in \N$, but there are exist $u_p = 0$,
then level
\[
  y = r_{c_1} + \sum_{k=2}^m r_{c_k} \prod_{i=1}^{k-1} u_{c_i}
    + r_p \prod_{i=1}^m u_{c_i}
\]
contains cylindrical closed interval $\Delta^E_{c_1c_2\ldots c_mp}$.
In this case any level is either point or closed interval (due to
continuity).

Situation is more complex if there are negative terms in the sequence
$u_n$.  Results of the previous section suggest that function $f$ does
not have a continuum level set, so there are not exist levels of the
function having fractal properties.  Moreover, the properties of the
level set essentially depend on the sequence $(u_n)$.

\begin{theorem}
  If there exist negative terms in the sequence $(u_n)$ and
  $E$-representation of number $x = \Delta^E_{g_1(x)g_2(x)\ldots
    g_n(x)\ldots}$ has the following property:
  \begin{equation}
    \label{eq:ugi.prod}
    u_{g_i(x)} u_{g_{i+1}(x)} < 0
  \end{equation}
  for infinite set of values $i \in \N$, then level $f^{-1}(y_0)$,
  where $y_0 = f(x)$, is a countable set.
\end{theorem}

\begin{proof}
  Using Corollary~\ref{cor:f.change} from Theorem~\ref{thm:f.maxmin}
  we have that changes in function $f$ have different signs on
  cylinders
  \[
    \Delta^E_{g_1(x)g_2(x)\ldots g_i(x)}
    \quad \text{and} \quad
    \Delta^E_{g_1(x)g_2(x)\ldots g_i(x)g_{i+1}(x)},
  \]
  and by Theorem~\ref{thm:f.maxmin} values of function $f$ on a
  cylinder form a closed interval such that its endpoints are values
  of function of cylinder's endpoints.  Thus, taking into account a
  continuity of function we have that line $y = y_0$ intersects the
  graph of function $f$ at least at two points belonging to cylinder
  $\Delta^E_{g_1(x)g_2(x)\ldots g_i(x)}$ and not belonging to cylinder
  $\Delta^E_{g_1(x)g_2(x)\ldots g_i(x)g_{i+1}(x)}$.

  The pattern repeats for the next $i$ such that
  condition~\eqref{eq:ugi.prod} holds.  This will repeats infinitely
  many times. So, level $f^{-1}(y_0)$ is an infinite set. It cannot be
  continuum as mentioned above, since the set of local maximums and
  minimums is countable.  Thus, $f^{-1}(y_0)$ is a countable set.
\end{proof}

\begin{remark}
  If condition~\eqref{eq:ugi.prod} holds, the $E$-representation of
  number $x$ does not have a simple period (i.e., period consisting of
  one symbol), and therefore, it is not an $E$-rational number.
\end{remark}

\section{``Symmetries'' of the graph. Scale invariance}

\begin{theorem}
  \label{thm:f.graph.symmetric}
  Graph $\Gamma_f$ of the function $f$ is a scale-invariant set,
  namely:
  \[
    \Gamma_f = \Delta^E_{(0)} \cup \bigcup_{i=0}^\infty \Gamma_i,
    \quad \text{where} \quad
    \Gamma_i = \varphi_i(\Gamma_f),
  \]
  and
  \[
    \varphi_i \colon \left\{
      \begin{aligned}
        x' &= \Delta^E_{ig_1(x)g_2(x)\ldots g_n(x)\ldots} = \delta_i(x), \\
        y' &= r_i + u_i f(x).
      \end{aligned}
    \right.
  \]
\end{theorem}

\begin{proof}
  1.~First of all we prove that
  \[
    \Gamma_f \subset \Delta^E_{(0)} \cup \bigcup_{i=0}^\infty \Gamma_i
    \equiv F.
  \]
  Let $\Delta^E_{(0)} \neq M(x,y) \in \Gamma_f$ that is $y = f(x)$.
  We show that $M \in F$ that is there exist $i$ such that $M \in
  \Gamma_i$.

  Let us consider point $M_1 \bigl( \Delta^E_{g_2(x)g_3(x)\ldots
    g_n(x)\ldots}, r_{g_2(x)} + u_{g_2(x)} f(\omega^2(x)) \bigr)$. It
  is evident that $M_1 \in \Gamma_f$. Then
  \[
    \varphi_{g_1(x)}(M_1)
    = M_1' \bigl( \Delta^E_{g_1(x)g_2(x)\ldots g_n(x)\ldots},
      r_{g_1(x)} + u_{g_1(x)}
      \bigl( r_{g_2(x)} + u_{g_2(x)} f(\omega^2(x)) \bigr) \bigr)
    \in \Gamma_{g_1(x)},
  \]
  and $M_1' = M$.  Then $M \in \Gamma_{g_1(x)}$, and thus $M \in F$
  and $\Gamma_f \subset F$.

  2.~Show that $F \subset \Gamma_f$. Let $M'(x',y') \in F$.  If
  $M'(x',y') = \Delta^E_{(0)}$, then there is nothing to prove;
  otherwise there exist $j$ such that $M' \in \Gamma_j$.  Consider
  point $M(x,y)$ such that $M' = \varphi_j(M)$.  Then $M \in \Gamma_f$
  that is $y = f(x)$, and $M'$ has coordinates
  \[
    \bigl( \Delta^E_{ig_1(x)g_2(x)\ldots g_n(x)\ldots},
      r_i + u_i f(\omega(x)) \bigr)
  \]
  that is $y' = f(x')$.

  Thus $M' \in \Gamma_f$ and $F \subset \Gamma_f$.

  Equality $\Gamma_f = F$ follows from inclusions $\Gamma_f \subset F$
  and $F \subset \Gamma_f$.
\end{proof}

\begin{remark}
  Transformation $\delta_i$ of left-open interval $(0, 1]$ defined by
  the formula
  \[
    x' = \delta_i(x) = \Delta^E_{ig_1(x)g_2(x)\ldots g_n(x)\ldots},
  \]
  has non-trivial properties, since
  \[
    x' = \frac{1}{2+i} + \frac{1}{2+i} \, x^*,
    \quad \text{where} \quad
    x^* = \Delta^E_{[i+g_1(x)]g_2(x)g_3(x)\ldots g_n(x)\ldots},
  \]
  and requires special study.
\end{remark}

\section{Integral properties of the function}

\begin{lemma}
  If $i$ is a fixed number belonging to $\Z_0$, then mapping
  \[
    x' = \delta_i(x) = \Delta^E_{ig_1(x)g_2(x)\ldots g_n(x)\ldots},
  \]
  is contractive with coefficient $\frac{1}{2+i}$.
\end{lemma}

\begin{proof}
  This proposition follows from the facts that image of cylinder
  $\Delta^E_{c_1c_2\ldots c_m}$ of rank $m$ under mapping $\delta_i$
  is cylinder $\Delta^E_{ic_1c_2\ldots c_m}$ of rank $m+1$ and the
  following relations hold:
  \begin{align*}
    \frac{\lvert\Delta^E_{ic_1c_2\ldots c_m}\rvert}
      {\lvert\Delta^E_{c_1c_2\ldots c_m}\rvert}
    &= \frac{(2+\sigma_1)(2+\sigma_2)\ldots(2+\sigma_m)(1+\sigma_m)}
       {(2+i)(2+\sigma_1+i)(2+\sigma_2+i)\ldots(2+\sigma_m+i)(1+\sigma_m+i)} = \\
    &= \frac{1}{2+i} \cdot \frac{1+\sigma_m}{1+\sigma_m+i}
      \cdot \prod_{k=1}^m \frac{2+\sigma_k}{2+\sigma_k+i}
    < \frac{1}{2+i}
    \leq \frac{1}{2},
  \end{align*}
  where $\sigma_k = c_1 + c_2 + \ldots + c_k$.
\end{proof}

\begin{corollary}
  Let $E$ be a set of zero Lebesgue measure.  Then measure of its
  image under mapping $\delta_i$ is equal to zero too.  That is
  \[
    \lambda(E) = 0 \; \Rightarrow \; \lambda(\delta_i(E)) = 0
  \]
  for any $i \in \Z_0$.
\end{corollary}

\begin{lemma}
  \label{lem:f.integral.sum}
  For Lebesgue integral the following equality holds:
  \[
    I \equiv \int_0^1 f(x) \, dx
    = \lim_{m\to\infty} \sum_{n=0}^m I_n
    = \sum_{n=0}^\infty I_n,
  \]
  where
  \[
    I_n = \int_{\Delta^E_{[n+1](0)}}^{\Delta^E_{n(0)}} f(x) \, dx.
  \]
\end{lemma}

\begin{proof}
  This proposition follows from integrability of the function,
  additive property of the Lebesgue integral, and fact that $(0, 1]$
  is a union of countable set of disjoint intervals
  $\bigl(\Delta^E_{[n+1](0)}, \Delta^E_{n(0)}\bigr]$, $n = 0$, $1$,
  $2$, \ldots.
\end{proof}

\begin{remark}
  ``Symmetries'' of graph of function studied in the previous section
  should help to express $I_n$ in terms of
  \[
    \int_0^1 f(x) \, dx.
  \]
  But non-self-similar geometry of $E$-representation essentially
  complicates this problem.
\end{remark}

\begin{theorem}
  For Lebesgue integral
  \[
    \int_0^1 f(x) \, dx
  \]
  the following estimate holds:
  \begin{equation}
    \label{eq:f.integral.est}
    \int_0^1 f(x) \, dx
    \leq \biggl( 1 - \sum_{n=0}^\infty \frac{u_n}{2+n} \biggr)^{-1}
      \sum_{n=0}^\infty \frac{r_n}{2+n}.
  \end{equation}
\end{theorem}

\begin{proof}
  Taking into account Theorem~\ref{thm:f.graph.symmetric} and
  Lemma~\ref{lem:f.solution.func.eqs} we have
  \[
    I_n = \int_{\Delta^E_{[n+1](0)}}^{\Delta^E_{n(0)}} f(x) dx
    = \int_0^1 \bigl(r_n + u_n f(x)\bigr) \, d\delta_n(x).
  \]
  Then from the Lemma~\ref{lem:f.integral.sum} it follows that
  \[
    d\delta_n(x) \leq \frac{1}{2+n} \, dx
  \]
  and
  \[
    I_n \leq \frac{1}{2+n} \int_0^1 \bigl(r_n + u_n f(x)\bigr) \, dx
    = \frac{r_n}{2+n} + \frac{u_n}{2+n} \int_0^1 f(x) \, dx.
  \]

  Thus
  \begin{align*}
    \int_0^1 f(x) \, dx
    &\leq \sum_{n=0}^\infty \biggl( \frac{r_n}{2+n}
      + \frac{u_n}{2+n} \int_0^1 f(x) \, dx \biggr) = \\
    &= \sum_{n=0}^\infty \frac{r_n}{2+n} + \biggl(
      \sum_{n=0}^\infty \frac{u_n}{2+n} \biggr) \int_0^1 f(x) \, dx
  \end{align*}
  and
  \[
    \biggl( 1 - \sum_{n=0}^\infty \frac{u_n}{2+n} \biggr) \int_0^1 f(x) \, dx
    \leq \sum_{n=0}^\infty \frac{r_n}{2+n}.
  \]
  This inequality is equivalent to~\eqref{eq:f.integral.est}.
\end{proof}

\providecommand{\bysame}{\leavevmode\hbox to3em{\hrulefill}\thinspace}
\providecommand{\MR}{\relax\ifhmode\unskip\space\fi MR }
\providecommand{\MRhref}[2]{%
  \href{http://www.ams.org/mathscinet-getitem?mr=#1}{#2}
}
\providecommand{\href}[2]{#2}

\end{document}